\documentclass[11pt, twoside, notitlepage]{amsart}

\usepackage{amsmath,amscd}
\usepackage{amssymb}
\usepackage{amsthm}
\usepackage{comment}
\usepackage{graphicx, xcolor}
\usepackage{todonotes}
\usepackage{mathrsfs}
\usepackage[ocgcolorlinks, linkcolor=blue]{hyperref}
\usepackage[framemethod=tikz]{mdframed}

\newcommand{\LC}{\left(}
\newcommand{\RC}{\right)}

\theoremstyle{plain}
\newtheorem{thm}{Theorem}[section]
\newtheorem{prop}{Proposition}[section]
\newtheorem{lemma}[prop]{Lemma}

\newtheorem{rmk}[prop]{Remark}

\numberwithin{equation}{section}
\newcommand {\R} {\mathbb{R}} 
 \newcommand {\N} {\mathbb{N}}
 
\newcommand {\p} {\partial}

\newcommand{\eps}{\epsilon}
\newcommand{\vareps}{\varepsilon}

\newcommand{\wt}{\widetilde}

\newcommand{\abs}[1]{\lvert #1 \rvert}          
\newcommand{\norm}[1]{\lVert #1 \rVert}         

\DeclareMathOperator{\F} {\mathcal{F}}

\definecolor{skyblue}{rgb}{0.85,0.85,1}

\author[]{}
\address{}
\curraddr{}
\email{}

\author[Lai]{Ru-Yu Lai}
\address{School of Mathematics, University of Minnesota, Minneapolis, MN 55455, USA}
\curraddr{}
\email{rylai@umn.edu}

\author[Ohm]{Laurel Ohm}
\address{Courant Institute, New York University, New York, NY 10012, USA}
\curraddr{}
\email{laurel.ohm@nyu.edu}


\title[Inverse problems for the fractional Laplace equation]{Inverse problems for the fractional Laplace equation with lower order nonlinear perturbations}

\begin{document}

\maketitle
\begin{abstract}
We study the inverse problem for the fractional Laplace equation with multiple nonlinear lower order terms. We show that the direct problem is well-posed and the inverse problem is uniquely solvable. More specifically, the unknown nonlinearities can be uniquely determined from exterior measurements under suitable settings.

	\medskip


\end{abstract}


\section{Introduction}
We study the inverse problem for the fractional Laplace equation with lower order nonlinear perturbations. The problem setup is as follows.
For $0<t<s<1$, let $\Omega \subset \R^n, n\geq 1$ be a bounded domain with smooth boundary $\p \Omega$, and  $\Omega_e:=\R^n\setminus \overline{\Omega}$ be the exterior domain of $\Omega$. 
We consider the following fractional elliptic equation:
\begin{align}\label{intro_eqn_MS}
\begin{cases}
(-\Delta)^s u+ q(x,u,\nabla^t u) + a(x,u) =0 & \hbox{ in } \Omega,\\
u=f  &  \hbox{ in } \Omega_e,\\
\end{cases}
\end{align}
where $a(x,u)$ is an unknown potential and the gradient term $q$ takes the form
\begin{equation}\label{defintion q}
q(x,u,\nabla^t u) := b(x)\int_{\R^n}\nabla^tu(x,y)\cdot\nabla^tu(x,y)\,dy + u^m(x)\int_{\R^n} d(x,y)\cdot\nabla^t u(x,y)\,dy 
\end{equation}
for integer $m\ge2$. Here the unknown scalar function $b(x)$ and vector-valued function $d(x,y)$, together with $a(x,u)$, are to be determined from the exterior measurement.

In \eqref{intro_eqn_MS}, the fractional Laplacian for $0<s<1$ is defined by 
\begin{align}\label{fractional Laplacian}
(-\Delta)^{s}u(x) := c_{n,s}\mathrm{P.V.}\int_{\mathbb{R}^{n}}\dfrac{u(x)-u(y)}{|x-y|^{n+2s}}dy,
\end{align}
for $u\in H^s(\mathbb R^n)$, where the symbol P.V. denotes the principal value and  
\begin{equation*}
c_{n,s}=\frac{\Gamma(\frac{n}{2}+s)}{|\Gamma(-s)|}\frac{4^{s}}{\pi^{n/2}} 
\end{equation*}
is a constant; see \cite{di2012hitchhiks} for the explicit expression. The space $H^s(\R^n)$ is the standard fractional Sobolev space; see also Section \ref{Sec 2}.
For $u\in H^s(\mathbb R^n)$, since $H^s(\R^n)\subset H^t(\R^n)$ for $0<t<s<1$, $u$ is also in $H^t(\R^n)$.  Then the \textit{fractional gradient} of $u$ at points $x$ and $y$ is defined by  
$$
    \nabla^t u(x,y):= {c^{1/2}_{n,t}\over \sqrt{2}} {y-x\over|x-y|^{n/2+t+1}}(u(x)-u(y)),
$$
and the linear operator $\nabla^t$ maps $H^t(\R^n)$ to $L^2(\R^{2n})$ \cite{Covi2020}. Further discussion of notation will appear in Section~\ref{Sec 2}.

For the coefficients $b(x)$ and $d(x,y)$, we assume that 
$b=b(x):\Omega \to \R$ and $d=d(x,y):\Omega\times\R^n\to \R^n$ satisfy
\begin{align}\label{condition bd}
b\in C(\Omega)\quad\hbox{ and }\quad d\in  C(\Omega\times\R^n)\, \hbox{ with compact support in $\Omega\times\Omega$.} 
\end{align} 
Furthermore, we assume that the coefficient $a=a(x,z):\Omega\times \R \to \R$ satisfies the following conditions:
\begin{align}\label{condition a}
\begin{cases}
\p^k_z a(x,0)=0\qquad \hbox{for all }x\in\overline\Omega, \; 0\le k\le  m \\
\hbox{the map $z\mapsto a(\cdot,z)$ is holomorphic with values in $C^s(\overline\Omega)$},
\end{cases}
\end{align}
where $C^s(\overline\Omega)$ denotes the usual H\"older space; see also Section \ref{Sec 2}.
Then the function $a$ can be expanded into the following power series:
\begin{align}\label{Taylor series a}
a(x,z) = \displaystyle\sum^\infty_{k=m+1} a_k(x){z^k\over k!}, \qquad  a_k(x):=\p^k_z a(x,0)\in C^s(\overline\Omega),
\end{align}
which converges in $C^s(\Omega\times\R)$ space.

The exterior measurement is encoded in the \emph{Dirichlet-to-Neumann} (DN) map:
\[
\Lambda: \widetilde{H}^s(\Omega_e)\to \big(\widetilde{H}^{s}(\Omega_e)\big)^\ast, \qquad \Lambda(f)=\left.(-\Delta)^s u\right|_{\Omega_e}, 
\]
where $u$ is the solution to \eqref{intro_eqn_MS} with exterior data $f$ and $\big(\widetilde{H}^{s}( \Omega_e)\big)^\ast$ represents the dual space of $\widetilde{H}^{s}(\Omega_e)$. 
For small data $f\in C^\infty_c(\Omega_e)$, we show in Section~\ref{Sec 2} that the problem \eqref{intro_eqn_MS} is well-posed and, therefore, we can define the DN map through the integral \eqref{DN integral} corresponding to the equation \eqref{intro_eqn_MS} and it is indeed well-defined. 

A fractional version of the well-known Calder\'on problem \cite{calderon, U2009} was first investigated in \cite{ghosh2016calder}, in which the authors studied the inverse problem for the linear fractional Schr\"odinger equation (with $q=0$ and $a(x,u)=a(x)u$ in \eqref{intro_eqn_MS}). Specifically, in \cite{ghosh2016calder} the potential $a(x)$ is uniquely determined from the associated DN map. The essential idea in obtaining this uniqueness result is to establish the strong uniqueness property of the fractional Laplacian $(-\Delta)^s$ (see Proposition~\ref{Prop:strong uniqueness}) and the associated Runge approximation property.
Since then, there have been many works concerning related inverse problems in various settings, including the problem with a single measurement \cite{GRSU18, Ruland2020single}, unique determination for the (anisotropic) fractional Laplacian and conductivity equation \cite{cekic2020calderon, covi2018, GLX}, stability estimates \cite{ruland2017fractional}, the inverse obstacle problem \cite{CLL2017simultaneously}, monotonicity inversion \cite{harrach2017nonlocal-monotonicity,harrach2020monotonicity}, nonlinear equations \cite{lai2019global, LaiL2020, Li2020_05, Lin202004}, fractional parabolic equations \cite{LLR2019calder}, fractional magnetic equations \cite{Covi2020, Li2019, Li2020magnetic}, higher order operators 	\cite{CMR2020, CMRU}, as well as equations with lower order nonlocal perturbations \cite{BGU18}.

\subsection{Main result}\label{subsec:uniqueness}
The main objective of this paper is to study the simultaneous reconstruction of three nonlinearities in a fractional equation. Due to the nonlocal property, this is by nature a partial data inverse problem. 
The main result of the paper is stated below.

\begin{thm}\label{Main Thm 1}
	Let $0<t<s<1$ and let $\Omega \subset \R^n$, $n\geq 1$ be a bounded domain with smooth boundary. Let $W_1,W_2$ be two arbitrary open sets in $\Omega_e$. Suppose that $b_j(x),\, d_j(x,y)$, and $a_j(x,z)$ each satisfy the conditions \eqref{condition bd} and \eqref{condition a} for $j=1,2$. 
    Suppose furthermore that 
	$$
	(d_1-d_2)(x,y)|x-y |^{-n/2-t}\in  	L^2(\Omega)\qquad \hbox{ for any fixed } x\in\Omega.
	$$ 
	Let $\Lambda_{j}(f)$ be the DN map corresponding to \eqref{intro_eqn_MS} with $a,b,d$ replaced by $a_j,b_j,d_j$, respectively, for $j=1,2$. Suppose that
	\begin{align}\label{DN map in Thm 1}
	  \left.	\Lambda_{1}(f) \right|_{W_2} = 	  \left.	\Lambda_{2}(f) \right|_{W_2} \qquad \text{ for any }f\in C^\infty_c(W_1)
	\end{align} 
	with $\norm{f}_{C^\infty_c(W_1)}<\varepsilon$, where $\varepsilon>0$ is sufficiently small.
	Then 
	$$b_1(x)=b_2(x)\qquad \hbox{ in } \Omega,$$
    $$d_1(x,y)\cdot(x-y)=d_2(x,y)\cdot(x-y)\qquad \hbox{ in } \Omega\times \Omega, $$ 
	and $$\quad a_1(x,z)=a_2(x,z) \qquad\hbox{ in } \Omega \times \R.$$
	\end{thm} 
 
\begin{rmk}
	We can fully recover the coefficient $d$ only if $d$ is of the form
	\[ d(x,y) = d_0(x,y)(x-y) \]
	for some scalar-valued function $d_0$ where $d_0(x,x)$ is known. This is due to the natural gauge enjoyed by equation \eqref{intro_eqn_MS}; see \cite{Covi2020}. In particular, if $u$ satisfies \eqref{intro_eqn_MS} with $d=\underline d(x,y)$, then $u$ also satisfies \eqref{intro_eqn_MS} for $d=\underline d(x,y)+d_\perp(x,y)$ for any $d_\perp$ satisfying $d_\perp\cdot(x-y)=0$. See also \cite{Covi2020}. 
\end{rmk}

The linearization scheme \cite{Isakov93} is a promising method for the study of inverse problem for local and nonlocal nonlinear elliptic equations. By performing a first order linearization of the DN map, one can reduce the inverse problem under study to the inverse problem for a linear equation. Then one can apply the available results for this linear case to recover the unknowns. The higher order linearization technique, in particular, uses nonlinearity as a tool in solving inverse problems for nonlinear equations. It involves introducing small parameters into the data, and then differentiating the nonlinear equation with respect to these parameters multiple times to obtain simpler linearized equations. 
Note that the application of this higher order linearization technique in treating local or nonlocal elliptic equations with power-type nonlinearities has been exploited in \cite{FO2019, KU201909, KU201905, LaiL2020, LaiZhou2020, LLLS201903, LLLS201905, Li2020_05, Lin202004}.

The inverse boundary value problem (IBVP) for nonlocal elliptic equations with nonlinearities was investigated in \cite{lai2019global, LaiL2020, Lin202004} for $(-\Delta)^s u+a(x,u)=0$. In particular, when $b=0,\,d=0$ in \eqref{intro_eqn_MS}, $a(x,u)$ is uniquely determined from the exterior measurement in \cite{lai2019global} based on first order linearization. The necessary condition $W_1=W_2$ in \cite{lai2019global} was removed later in \cite{LaiL2020}, which also showed the well-posedness of the equation using higher order linearization.
Moreover, in \cite{Li2020_05}, the problem for the nonlinear fractional magnetic equation was studied by applying first order linearization. 

We shall next discuss the IBVP for local nonlinear elliptic equations. This problem has been extensively studied in the literature. For instance, $-\Delta u+a(x,u)=0$ was studied in \cite{victorN, isakov1994global, sun2010inverse} for the full data problem and \cite{KU201905, LLLS201905} for the partial data setting when $n\geq 2$. The quasilinear equation $-\Delta u+a(u,\nabla u)=0$ was studied in \cite{victor01} when $n=3$ and $-\Delta u+a(x,\nabla u)=0$ was investigated in \cite{sun2004inverse} when $n=2$. 
It was however noted in \cite{sun2004inverse} that the uniqueness of recovery of more general nonlinearity $a(x,u,\nabla u)$ in $-\Delta u+a(x,u, \nabla u)=0$ in general fails.
We refer the interested reader to \cite{CNV19, Sun2002, Isakov93, Kang2002, Sun96, SunUhlmann97} for related results.

In this paper, we apply the higher order linearization technique to prove the well-posedness of \eqref{intro_eqn_MS} and reconstruct the unknown coefficients when the data is sufficiently small ($\norm{f}_{C^\infty_c(W_1)}<\varepsilon$ for some $\varepsilon>0$). 
More specifically, in our setting, differentiating \eqref{intro_eqn_MS} w.r.t. to the small parameter $\varepsilon$ yields the equation $(-\Delta)^su^{(1)}=0$, whose solution is independent of unknown coefficients. Differentiating \eqref{intro_eqn_MS} twice leads to $(-\Delta)^s u^{(2)} + b(x)h(x;u^{(1)})=0$, which specifically contains only the unknown $b$ with $h(x;u^{(1)})$ acting as a source term. We can then determine $b$ uniquely from the exterior data; see Section~\ref{Sec 3} for notation and details. Finally, let us remark that the nonlinearities here indeed help by reducing the nonlinear equation to $(-\Delta)^su^{(1)}=0$ after the first linearization. This then enables the use of both strong uniqueness property (Proposition~\ref{Prop:strong uniqueness}) and the Runge approximation property for $(-\Delta)^s$.  

As mentioned above, when $s=1$, $a(x,u, \nabla u)$ in $-\Delta u+a(x,u, \nabla u)=0$ cannot be fully determined in general, which inspires us to consider the nonlocal setting as in \eqref{intro_eqn_MS}. We may think of the three nonlinear terms in \eqref{intro_eqn_MS} as an example of the general nonlinear term $a(x,u, \nabla^t u)$. We show that they can be recovered simultaneously in Theorem~\ref{Main Thm 1}.

Finally, for the local equations, when $s=1$, the determination of multiple nonlinear terms was investigated in \cite{KU201909} for $-\Delta u+q(x)\nabla u\cdot\nabla u + a(x,u)=0$ and in \cite{LaiZhou2020} for the magnetic Schr\"odinger equation with nonlinear terms like $a_1(x,u)+a_2(u,\nabla u)$.   
Both \cite{KU201909} and \cite{LaiZhou2020} applied the higher order linearization and the density result for harmonic functions to solve the inverse problem. 
Here we apply an analogous density result, the Runge approximation, characterizing the density of the collection of solutions to the fractional Laplace equation in $L^2$ space. This density result is crucial to recovering the coefficient $d$; see Section~\ref{Sec 3} for details.

The paper is organized as follows. Section~\ref{Sec 2} introduces notation and several previous results, including the unique continuation property and the maximum principle.  
The well-posedness result for \eqref{intro_eqn_MS} is also stated and proven in Section~\ref{Sec 2}. Finally in Section~\ref{Sec 3} we use the results of Section~\ref{Sec 2} to show Theorem~\ref{Main Thm 1}.

\section{Preliminaries}\label{Sec 2}

In this section, we introduce notation and the well-posedness result for the problem \eqref{intro_eqn_MS}.

\subsection{Function spaces}
We starting by defining the H\"older spaces. Let $U\subset\R^n$ be an open set and $k$ a nonnegative integer. For a given $0<\alpha <1$, the H\"older space $C^{k,\alpha}(U)$ is defined by
\[
C^{k,\alpha}(U):=\left\{f:U\to \R:\ \norm{f}_{C^{k,\alpha}(U)}<\infty \right\},
\]
where 
\[
\norm{f}_{C^{k,\alpha}(U)}:=\sum_{|\beta|\leq k}\norm{\p ^\beta f}_{L^\infty(U)}+\sup_{x\neq y, \ x,y\in U}\sum_{|\beta|=k}\frac{|\p ^\beta f(x)-\p^\beta f(y)|}{|x-y|^\alpha}.
\]
Here $\beta=(\beta_1,\ldots,\beta_n)$ is a multi-index with $\beta_i \in \N^+ \cup \{0\}$ and $|\beta|=\beta_1 +\ldots +\beta_n$. When $k=0$, we simply set $C^\alpha(U) \equiv C^{0,\alpha}(U)$. 
We use $C_c^{k}(U)$ to denote the space of functions on $C^{k}(U)$ with compact support in $U$. Note that the above notation applies similarly for the closed set $\overline{U}$. 

Next, following the notation in \cite{ghosh2016calder}, for $0<s<1$, we use $H^{s}(\mathbb{R}^{n}):=W^{s,2}(\mathbb{R}^{n})$ to denote the $L^{2}$-based Sobolev space with the following norm: 
\begin{equation}\notag
\|u\|_{H^{s}(\mathbb{R}^{n})}^{2}\\=\|u\|_{L^{2}(\mathbb{R}^{n})}^{2}+\|(-\Delta)^{s/2}u\|_{L^{2}(\mathbb{R}^{n})}^{2}.\label{eq:H^s norm}
\end{equation}
Here, by the Parseval identity, the semi-norm $\|(-\Delta)^{s/2}u\|_{L^{2}(\mathbb{R}^{n})}^{2}$
can be expressed as 
\[
\|(-\Delta)^{s/2}u\|_{L^{2}(\mathbb{R}^{n})}^{2}=\left((-\Delta)^{s}u,u\right)_{\mathbb{R}^{n}},
\]
where the operator $(-\Delta)^s$ is as defined in \eqref{fractional Laplacian}.

For scalar $\beta\in\mathbb{R}$, we define the following Sobolev spaces: 
\begin{align*}
H^{\beta}(U) & :=\left\{u|_{U}:\, u\in H^{\beta}(\mathbb{R}^{n})\right\},\\
\widetilde{H}^{\beta}(U) & :=\text{closure of \ensuremath{C_{c}^{\infty}(U)} in \ensuremath{H^{\beta}(\mathbb{R}^{n})}},\\
H_{0}^{\beta}(U) & :=\text{closure of \ensuremath{C_{c}^{\infty}(U)} in \ensuremath{H^{\beta}(U)}},
\end{align*}
and 
\[
H_{\overline{U}}^{\beta}(\R^n):=\left\{u\in H^{\beta}(\mathbb{R}^{n}):\,\mathrm{supp}(u)\subset\overline{U}\right\}.
\]
The Sobolev space $H^{\beta}(U)$ is complete under the graph norm
\[
\|u\|_{H^{\beta}(U)}:=\inf\left\{ \|v\|_{H^{\beta}(\mathbb{R}^{n})}:\,v\in H^{\beta}(\mathbb{R}^{n})\mbox{ and }v|_{U}=u\right\} .
\]
It is known that $\widetilde{H}^{\beta}(U)\subsetneq H_{0}^{\beta}(U)$,
and $H_{\overline{U}}^{\beta}(\R^n)$ is a closed subspace of $H^{\beta}(\mathbb{R}^{n})$. Moreover, 
$$
    (H^\beta(U))^\ast = \widetilde{H}^{-\beta}(U),\qquad (\widetilde{H}^\beta(U))^\ast = H^{-\beta}(U),\quad \beta\in\R.
$$
If $U$ is also a bounded Lipschitz domain, the dual spaces can be expressed as 
\begin{align*}
	H^\beta_{\overline{U}}(\R^n)=\widetilde{H}^{\beta}(U),\quad \text{ and }\quad (H^\beta_{\overline{U}}(\R^n))^\ast = H^{-\beta}(U), \quad \text{ and }\quad (H^\beta(U))^\ast=H^{-\beta}_{\overline U}(\R^n).
\end{align*}
For more details on fractional Sobolev spaces, we refer to \cite{di2012hitchhiks, ghosh2016calder, mclean2000strongly}.

\subsection{Well-posedness} 
Let $0<t<s<1$ and let $\Omega\subset \R^n$, $n\geq 1$ be a bounded domain with smooth boundary $\p \Omega$. We consider the following Dirichlet problem with exterior data:
\begin{align}\label{eqn:S}
\begin{cases}
(-\Delta)^s u+ q(x,u,\nabla^t u) + a(x,u) =0 & \hbox{ in } \Omega,\\
u=f  &  \hbox{ in } \Omega_e,\\
\end{cases}
\end{align}
where $f\in C^\infty_c(\Omega_e)$, and $q$ and $a$ are as in \eqref{defintion q} and \eqref{Taylor series a}.  

For notational brevity, we define the function $h$ as
$$
h(x;u,v):=\int_{\R^n}\nabla^t u(x,y)\cdot\nabla^t v(x,y)\,dy,
$$
and, in particular, when $u=v$, we denote
\begin{align}\label{definition h}
h(x;u):=\int_{\R^n}\nabla^t u(x,y)\cdot\nabla^t u(x,y)\,dy.
\end{align}
We also define
\begin{align}\label{definition psi}
 \psi(x;d,u):= u^m(x)\int_{\R^n} d(x,y)\cdot\nabla^t u(x,y) \, dy.
\end{align}
Then $q$ can be expressed as
$q(x,u,\nabla^t u) = b(x)h(x;u)+\psi(x;d,u).$

In the following lemma, we analyze the boundness of $h$ and $\psi$, which will be a crucial ingredient in proving the well-posedness result.
\begin{lemma}\label{lemma gradient estimate}
	Let $0<t<s<1$ and $u, v\in C^s(\R^n)$. For a fixed constant $R>0$, we have
    \begin{align}\label{gradient bound 2}
    \int_{\R^n} {|(u(x)-u(y))(v(x)-v(y))|\over|x-y|^{n+2t}}\,dy\leq  C_n\|u\|_{C^s(\R^n)}\|v\|_{C^s(\R^n)} \LC{1\over 2s-2t}R^{2s-2t} + {2\over  t}R^{-2t}\RC  
    \end{align}
	for all $x\in\overline\Omega$. In particular, when $u=v$, we have
	\begin{align}\label{gradient bound 1}
	 \int_{\R^n} {|u(x)-u(y)|^2\over|x-y|^{n+2t}}\,dy\leq  C_n\|u\|^2_{C^s(\R^n)} \LC{1\over 2s-2t}R^{2s-2t} + {2\over  t}R^{-2t}\RC  
	\end{align}
	for all $x\in\overline\Omega$. Here the constant $C_n$ only depends on $n$.
\end{lemma}

\begin{proof}
We first denote $M:=\|u\|_{C^s(\R^n)}$ and $\wt M:=\|v\|_{C^s(\R^n)}$  and note that $u,v\in C^s(\R^n)$ yields
\begin{align}\label{Lipschtz}
 |u(x)-u(y)|\leq M|x-y|^s,\qquad |v(x)-v(y)|\leq \wt M|x-y|^s
\end{align}
for all $x,\,y\in\R^n$.

To show \eqref{gradient bound 2}, we note that for any fixed $x\in\overline\Omega$ we have
\begin{align*}
    &\hskip.5cm \int_{\R^n} {|u(x)-u(y)||v(x)-v(y)|\over|x-y|^{n+2t}}\,dy \\
	&= \int_{|x-y|\leq R} {|u(x)-u(y)||v(x)-v(y)|\over|x-y|^{n+2t}}\,dy+\int_{|x-y|> R}  {|u(x)-u(y)||v(x)-v(y)|\over|x-y|^{n+2t}}\,dy\\
	&\leq M\wt M \int_{|x-y|\leq R}  |x-y|^{-n-2t+2s} \,dy + (2M)(2\wt M)	\int_{|x-y|> R}  |x-y|^{-n-2t} \,dy.
\end{align*}
Here we used \eqref{Lipschtz} to derive the first term in the inequality.
Applying a change of variables to spherical coordinates and recalling that $t<s$, we then obtain
\begin{align*}
\int_{\R^n} {|u(x)-u(y)||v(x)-v(y)|\over|x-y|^{n+2t}}\,dy &\leq C_n M\wt M \int_{0}^R  \rho^{2s-2t-1} \,d\rho + C_n4 M\wt M\int_{R}^\infty \rho^{-2t-1} \,dy\\
 &= C_n\|u\|_{C^s(\R^n)}\|v\|_{C^s(\R^n)} \LC{1\over 2s-2t}R^{2s-2t} +  {2\over t}R^{-2t}\RC,
\end{align*}
which completes the proof of \eqref{gradient bound 2}.
Finally, the estimate \eqref{gradient bound 2} implies \eqref{gradient bound 1} when $u=v$. 
\end{proof}
 
We note that Lemma \ref{lemma gradient estimate} implies that
\begin{align}\label{bound h}
    \|h(x;u,v)\|_{L^\infty(\Omega)}&=\| \int_{\R^n}\nabla^tu(x,y)\cdot\nabla^tv(x,y)\,dy \|_{L^\infty( \Omega)} \notag\\
    &\leq  \|{c_{n,t}\over 2}\int_{\R^n}{|u(x)-u(y)||v(x)-v(y)|\over|x-y|^{n+2t}}\,dy \|_{L^\infty( \Omega)} \notag\\
    &\leq  C\|u\|_{C^s(\R^n)}\|v\|_{C^s(\R^n)}\LC{1\over 2s-2t}R^{2s-2t} +   {2\over  t}R^{-2t}\RC,
\end{align}
where the constant $C$ depends on $n$ and $t$, and thus
$
    h(x;u,v)
$
is in $L^\infty(\Omega)$. 

Similarly, Lemma~\ref{lemma gradient estimate} also implies that 
\begin{align}\label{bound psi}
\|\psi(x;d,u) \|_{L^\infty(\Omega)} &\le \norm{u}_{L^\infty(\Omega)}^m \norm{ \int_{\R^n}\abs{d(\cdot,y)}^2 \, dy}_{L^\infty(\Omega)}^{1/2} \|\int_{\R^n}\abs{ \nabla^t u(x,y)}^2 \, dy \|_{L^\infty(\Omega)}^{1/2} \notag\\
&\le C  \norm{u}_{L^\infty(\Omega)}^m \|\int_{\R^n} {|u(x)-u(y)|^2\over|x-y|^{n+2t}}\,dy \|_{L^\infty(\Omega)}^{1/2}\notag\\
&\le C \norm{u}_{C^s(\R^n)}^{1+m} \LC{1\over 2s-2t}R^{2s-2t} + {2\over t}R^{-2t}\RC^{1/2}.
\end{align}
Here $C$ depends on $\Omega,n,t$, and the coefficient $d$. \\

\begin{rmk}
	Lemma~\ref{lemma gradient estimate} suggests that in order to have pointwise control on the terms $h(x;u)$ and $\psi(x;d,u)$, we must consider $t$ satisfying $0<t<s<1$, as the above arguments fail when $t=s$. 
\end{rmk}

The following lemma will also be used in showing the contraction property in the proof of Theorem~\ref{Thm:well posedness}.
\begin{lemma}\label{lemma:discrepancy} Let $0<t<s<1$ and $u_1,u_2\in C^s(\R^n)$. We have the following two estimates:
$$
    \|h(x;u_1) - h(x;u_2)\|_{L^\infty(\Omega)} \leq C\|u_1-u_2\|_{C^s(\R^n)}\|u_1+u_2\|_{C^s(\R^n)} 
$$
and
\begin{align*}
    &\hskip.5cm \|\psi(x;d,u_1) - \psi(x;d,u_2)\|_{L^\infty(\Omega)} \\
    & \leq C\|u_1-u_2\|_{C^s(\R^n)}\LC  \|u_1\|_{C^s(\R^n)} \sum_{k=1}^m \|u_1\|_{C^s(\R^n)}^{m-k} \|u_2\|_{C^s(\R^n)}^{k-1} + \|u_2\|^m_{C^s(\R^n)} \RC.
\end{align*}
Here the constant $C$ depends only on $n,t,s,d$, and $\Omega$.
\end{lemma}
\begin{proof}
First, from the definition of $h$ and \eqref{bound h} with $R=1$, we derive
\begin{align*}
h(x;u_1) - h(x;u_2) &= \int_{\R^n}\nabla^tu_1(x,y)\cdot\nabla^tu_1(x,y)\,dy - \int_{\R^n}\nabla^tu_2(x,y)\cdot\nabla^tu_2(x,y)\,dy \\
&= \int_{\R^n}(\nabla^tu_1 - \nabla^tu_2) \cdot (\nabla^tu_1 + \nabla^tu_2)(x,y) \,dy \\
&=h(x; u_1-u_2, u_1+u_2)\\
&\leq  C\|u_1-u_2\|_{C^s(\R^n)}\|u_1+u_2\|_{C^s(\R^n)},
\end{align*}
for any $x\in\Omega$, where $C$ is a constant depending on $s,t$ and $n$. 

Next, for any $x\in\Omega$, we consider
\begin{align*}
&\hskip.5cm \psi(x;d,u_1) - \psi(x;d,u_2) \\
&= u_1^m(x)\int_{\R^n} d(x,y)\cdot\nabla^t u_1(x,y)\,dy - u_2^m(x)\int_{\R^n} d(x,y)\cdot\nabla^t u_2(x,y) \, dy\\
&= (u_1^m(x)-u_2^m(x))\int_{\R^n} d(x,y)\cdot\nabla^t u_1(x,y)\,dy\\
&\quad + u_2^m(x) \LC \int_{\R^n} d(x,y)\cdot\nabla^t u_1(x,y)\,dy - \int_{\R^n} d(x,y)\cdot\nabla^t u_2(x,y)\,dy \RC\\
&= (u_1(x)-u_2(x))\LC \sum_{k=1}^m u_1^{m-k}u_2^{k-1} \RC \int_{\R^n} d(x,y)\cdot\nabla^t u_1(x,y)\,dy\\
&\quad + u_2^m(x) \LC \int_{\R^n} d(x,y)\cdot\nabla^t (u_1- u_2)(x,y)\,dy \RC.
\end{align*}
Application of a similar argument as in \eqref{bound psi} gives the upper bound for the following terms:
$$
    \int_{\R^n} d(x,y)\cdot\nabla^t u_1(x,y)\,dy \leq C \|u_1\|_{C^s(\R^n)}
$$
and
$$
u_2^m(x) \LC \int_{\R^n} d(x,y)\cdot\nabla^t (u_1- u_2)(x,y)\,dy \RC\leq C \|u_2\|^m_{C^s(\R^n)} \|u_1-u_2\|_{C^s(\R^n)}.
$$
Combining these estimates, we obtain the desired estimate for $\psi$.

\end{proof}

We are now ready to show the well-posedness result.
\begin{thm}[Well-posedness]\label{Thm:well posedness}
Let $0<t<s<1$ and let $\Omega\subset \R^n$, $n\geq 1$ be a bounded domain with smooth boundary $\p\Omega$. 
Suppose that $b(x),\, d(x,y)$, and $a(x,z)$ satisfy the conditions \eqref{condition bd} - \eqref{Taylor series a}.
Then there exists a small parameter $0<\varepsilon<1$ such that when
\begin{align}\label{small boundary}
f\in \mathcal{X}:=\left\{f\in C^\infty_c(\Omega_e) :\ \norm{f}_{C^\infty_c(\Omega_e)}\leq \vareps \right\},
\end{align} 
the boundary value problem \eqref{eqn:S} has a unique small solution $u \in C^s(\R^n)\cap H^s(\R^n)$. Moreover, the solution $u$ satisfies the estimate
$$
    \|u\|_{C^{s}(\R^n)} \leq C \|f\|_{C^\infty_c(\Omega_e)},
$$ 
where the constant $C>0$ is independent of $u$ and $f$.
\end{thm}

\begin{proof}
	Suppose that $\|f\|_{C^\infty_{c}(\Omega_e)}\leq \varepsilon$ for some sufficiently small $\vareps>0$. We may extend $f$ to the whole space $\R^n$ by zero so that $\|f\|_{C^\infty_{c}(\R^n)}\leq \varepsilon$.   

Before getting into the proof, we recall the following result of \cite{ghosh2016calder}. For $g\in L^\infty(\Omega)$, there exists a unique solution $\tilde{v}\in H^s(\R^n)$ to the problem
\begin{align}\label{zero boundary value problem} 
\begin{cases}
(-\Delta)^s \tilde{v} =g & \text{ in }\Omega, \\
\tilde{v}=0 & \text{ in }\Omega_e.
\end{cases}
\end{align}
Moreover, by \cite[Proposition~1.1]{ros2014dirichlet}, we have
\begin{align*}
\|\tilde{v}\|_{C^{s}(\R^n)}\leq C\|g\|_{L^\infty(\Omega)}
\end{align*}
for some constant $C>0$ depending on $s$ and $\Omega$. This enables us to define the solution operator $$\mathcal L_s^{-1}: g\in L^\infty(\Omega) \rightarrow \tilde{v}\in C^s(\R^n)\cap H^s(\R^n)$$ to \eqref{zero boundary value problem}. The solution $\mathcal{L}_s^{-1}(g)$ to \eqref{zero boundary value problem} then satisfies
\begin{align}\label{Ros-Oton estimate}
\|\mathcal L_s^{-1}(g)\|_{C^{s}(\R^n)}\leq C\|g\|_{L^\infty(\Omega)}.
\end{align}

We may now proceed to the linearization procedure.

\noindent\textbf{Step 1: The linearized problem.}
We first consider the linear part of \eqref{eqn:S}, given by
\begin{align}\label{eqn:u0}
\begin{cases}
(-\Delta)^s u_0 =0 & \hbox{ in } \Omega,\\
u_0=f  &  \hbox{ in } \Omega_e.\\
\end{cases} 
\end{align}
Due to \cite{ghosh2016calder}, there exists a unique solution $u_0 \in H^s(\R^n)$ to \eqref{eqn:u0}. By considering $(-\Delta)^s(u_0-f) = -(-\Delta)^sf$ with $(u_0-f)|_{\Omega_e}=0$, we may then apply \eqref{Ros-Oton estimate} to obtain
\begin{align}\label{uniform bound of u_0}
	\norm{u_0}_{C^s(\R^n)}\leq C\norm{f}_{C_c^\infty(\Omega_e)},
\end{align} 
where the constant $C>0$ depends only on $s$ and $\Omega$.

We next consider $v:=u-u_0$, where $u_0$ satisfies \eqref{eqn:u0} and $u$ satisfies the original nonlinear equation \eqref{eqn:S}. If such a function $v$ exists, then $v$ satisfies the following problem:
\begin{align}\label{eqn:v}
\begin{cases}
(-\Delta)^s v = G(v) & \hbox{ in } \Omega,\\
v=0  &  \hbox{ in } \Omega_e,
\end{cases}
\end{align}
where $G(\phi)$ is defined by
\begin{align*}
 G(\phi):= -b(x) h(x;u_0+\phi) - \psi(x;d,u_0+\phi)-a(x,u_0 +\phi).
\end{align*}
The problem is now reduced to showing the unique existence of a solution $v$ to \eqref{eqn:v}. To this end, we will construct a contraction map and establish the unique existence of a solution by the contraction mapping principle.\\

\noindent\textbf{Step 2: Construct a contraction map.} 
Let us define the set 
$$\mathcal{M} = \left\{\phi\in C^{s}(\R^n):\ \phi|_{\Omega_e}=0,\ \|\phi\|_{C^{s}(\R^n)}\leq \delta \right\},$$
where $0 < \delta < 1$ will be determined later (by choosing sufficiently small $\delta$ to satisfy the specific inequalities below). It is easy to see that $\mathcal{M}$ is a Banach space.

We define the map $\mathcal{F}$ on $\mathcal{M}$ by
$$\mathcal{F}:=\mathcal{L}_s^{-1}\circ G.$$ 
We will show below that $\mathcal{F}$ is indeed a contraction map on $\mathcal{M}$.

We first claim that $\F: \mathcal{M}\to \mathcal{M}$.
By \eqref{bound h}, \eqref{bound psi}, \eqref{Ros-Oton estimate}, and the Taylor expansion of $a$ \eqref{Taylor series a}, for any $\phi \in \mathcal{M}$, we obtain
$\mathcal{F}(\phi)\in  C^{s}(\R^n)\cap H^s(\R^n)$, and 
\begin{align}\label{F:M to M}
   \norm{\F(\phi)}_{C^{s}(\R^n)}
  &\leq C\|G(\phi)\|_{L^\infty(\Omega)} \notag\\
  &= C  \|b(x) h(x;u_0+\phi) + \psi(x; d,u_0+\phi)+a(x,u_0 +\phi)\|_{L^\infty(\Omega)}  \notag\\
  &\leq C \|b\|_{C (\Omega )}\|u_0+\phi\|_{C^s(\R^n)}^2 +C \|u_0+\phi\|_{C^s(\R^n)}^{m+1} + C \|u_0+\phi\|^{m+1}_{C^s(\overline\Omega)}  \notag\\
  & \leq  C \|b\|_{C (\Omega )}(\delta+\vareps)^2 + C (\delta+\vareps)^{m+1} +C(\delta+\vareps)^{m+1},
\end{align} 
where the constant $C$ depends on $s,t,n$ and $\Omega$. This indicates that the function $G(\phi)\in L^\infty(\Omega)$.
Choosing sufficiently small $\varepsilon,\delta$, we then have 
\[
  \norm{\F(\phi)}_{C^{s}(\R^n)} \leq C (\varepsilon+\delta)^2 +  C (\varepsilon+\delta)^{1+m} +C(\delta+\vareps)^{m+1}< \delta,
\]  
which yields that $\F$ maps $\mathcal{M}$ into itself.

We also need to show that $\F$ is contractive. For any $\phi_1,\phi_2 \in \mathcal{M}$, we apply Lemma~\ref{lemma:discrepancy}, \eqref{Taylor series a}, and \eqref{Ros-Oton estimate} to get
\begin{align}\label{some estimate 1}
 \norm{\F(\phi_1)-\F(\phi_2)}_{C^{s}(\R^n)}   &=  \|(\mathcal{L}_s^{-1}\circ G)(\phi_1) -(\mathcal{L}_s^{-1}\circ G)(\phi_2)\|_{C^{s}(\R^n)}   \notag \\
&\leq  C \|G(\phi_1) -G(\phi_2) \|_{L^\infty(\Omega)}  \notag \\
&\leq  C\|b(x) (h(x;u_0+\phi_1)- h(x;u_0+\phi_2))\|_{L^\infty(\Omega)} \notag\\
&\quad+ C \|\psi(x;d,u_0 +\phi_1)-\psi(x;d,u_0 +\phi_2)\|_{L^\infty(\Omega)}  \notag\\
&\quad + C \|a(x,u_0 +\phi_1)-a(x,u_0 +\phi_2)\|_{L^\infty(\Omega)}  \notag\\
&\leq C  (\varepsilon+\delta)\|\phi_1-\phi_2\|_{C^s(\R^n)} + C  (\vareps +\delta)^m \|\phi_1 -\phi_2\|_{C^{s}(\R^n)} \notag\\
&\quad + C(\varepsilon +\delta)^{m} \|\phi_1 -\phi_2\|_{C^{s}(\R^n)},
\end{align} 
where $C$ is independent of $\varepsilon,\,\delta$.

By further taking $\varepsilon,\, \delta$ sufficiently small so that $C(\varepsilon+\delta) + C (\varepsilon+\delta)^m + C(\varepsilon+\delta)^m <1$, the following estimate also holds: 
$$
    \norm{\F(\phi_1)-\F(\phi_2)}_{C^{s}(\R^n)}<\|\phi_1 -\phi_2\|_{C^{s}(\R^n)}.
$$
Combining these results, we have shown that $\F$ is a contraction mapping on $\mathcal{M}$.

Finally, the contraction mapping principle gives that there is a fixed point $v\in \mathcal{M}$ such that $\mathcal{F}(v)=v$ and thus $v\in H^s(\R^n)$ as well. 
This $v$ is the solution to the equation \eqref{eqn:v} and also satisfies
\begin{align}\label{estimate v}
\|v\|_{C^{s}(\R^n)}\leq  C (\|u_0\|^2_{C^s(\overline\Omega)} + \|v\|^2_{C^{s}(\overline\Omega)})\leq  C\LC \varepsilon \|f\|_{C^\infty_c(\Omega_e)} + \delta\|v\|_{C^{s}(\overline\Omega)}  \RC 
\end{align} 
due to \eqref{F:M to M}.
For $\delta$ small enough, by absorbing $C\delta\|v\|_{C^{s}(\overline\Omega)}$ into the left-hand side of \eqref{estimate v}, we then have 
\begin{align*} 
\|v\|_{C^{s}(\R^n)}
\leq  C \varepsilon \|f\|_{C^\infty_c(\Omega_e)}.
\end{align*} 
As a result, we obtain the solution $u=u_0+v \in C^{s}(\R^n)$ to \eqref{eqn:S} and it satisfies
\begin{align*} 
\|u\|_{C^{s}(\R^n)}
\leq  C \|f\|_{C^\infty_c(\Omega_e)}
\end{align*} 
for some constant $C>0$ independent of $u$ and $f$. This completes the proof of well-posedness for the boundary value problem \eqref{eqn:S}.
\end{proof}

\subsection{The DN map}
In this subsection, we will define the corresponding DN map for the equation \eqref{eqn:S}.

By Theorem~\ref{Thm:well posedness}, for $f\in \mathcal{X}$, there exists a unique (small) solution $u_f\in C^s(\R^n)\cap H^s(\R^n)$ to \eqref{eqn:S} with the exterior data $u_f|_{\Omega_e}=f$. We define the DN map as follows:
\begin{equation}\label{DN integral}
\begin{aligned}
\left\langle \Lambda (f),\varphi\right\rangle  := \int_{\R^n}(-\Delta)^{s/2}u_f (-\Delta)^{s/2}\varphi \,dx+ \int_{\Omega}q(x,u_f,\nabla^t u_f)  \varphi 
 + a(x,u_f)\varphi \, dx
\end{aligned}
\end{equation}
for $\varphi\in \widetilde{H}^s(\Omega_e)$, where $q$ and $a$ are as defined in \eqref{defintion q} and \eqref{Taylor series a}.
Note that \eqref{DN integral} is not a bilinear form as in \cite{ghosh2016calder} due to the nonlinear terms $q$ and $a$.

\begin{prop}\label{prop:DNmap} 
	 Let $\Omega\subset \R^n$ be a bounded domain with smooth boundary $\p \Omega$ for $n\geq 1$, $0<t<s<1$. 
	 Suppose that $b,\, d$, and $a=a(x,z)$ satisfy the conditions \eqref{condition bd} - \eqref{Taylor series a}.
	  Then the DN map
	\[
	\Lambda: \mathcal{X}\subset \widetilde{H}^{s}(\Omega_e) \to \big( \widetilde{H}^{s}(\Omega_e)\big)^{*} 
	\]
	is bounded and satisfies
	\begin{equation}
	\left.\Lambda (f)\right|_{\Omega_e}=\left. (-\Delta)^s u_f\right|_{\Omega_e}.
	\end{equation}
\end{prop}
\begin{proof}
For small $f\in C^\infty_c(\Omega_e)$, due to Theorem~\ref{Thm:well posedness}, there exists a unique solution $u_f$ to \eqref{eqn:S}. Following the discussion in \cite[Section 3]{ghosh2016calder}, we may apply the Parseval identity to \eqref{DN integral}. Then for arbitrary $\varphi\in \widetilde{H}^s(\Omega_e)$, we obtain that
	\begin{align*}
 	& \int_{\R^n} ((-\Delta)^{s/2}u_f (-\Delta)^{s/2}\varphi \,dx + \int_\Omega q(x,u_f,\nabla^t u_f) \varphi
	+ a(x,u_f) \varphi \, dx\\
	=\, & \int_{\Omega_e}(-\Delta)^{s}u_f \varphi \, dx, 
	\end{align*}
	by using \eqref{eqn:S}.
	Thus $\Lambda (f)=(-\Delta)^s u_f$ in $\Omega_e$. 
\end{proof}

\subsection{Known results}
We state two known results which are crucial in the proof of Theorem~\ref{Main Thm 1}.

The first is the unique continuation property (UCP) for the fractional Laplacian \cite[Theorem~1.2]{ghosh2016calder}.
\begin{prop}[UCP]\label{Prop:strong uniqueness}
	Suppose that $U$ is a nonempty open subset of $\R^n$, $n\geq 1$. Let $0<s<1$ and $v\in H^{r}(\R^n)$ for $r\in\R$.
	If $v=(-\Delta)^s v=0$
	in some open set $U$ of $\R^n$, then $v\equiv0$ in $\mathbb{R}^{n}$. 
\end{prop}

The second result is the maximum principle for the fractional Laplacian. The proof of the following proposition can be found in \cite{LaiL2020} and \cite{lai2019global}, which extends the result in \cite{ros2015nonlocal} to include a nonzero potential term. 
\begin{prop}[Maximum principle]\label{Prop: strong max principle}
	Let $\Omega\subset \R^n$, $n\geq 1$ be a bounded domain with $C^1$ boundary $\p\Omega$, and $0<s<1$. Suppose that $w(x)\in L^\infty(\Omega)$ be a nonnegative potential. Let $u\in H^s(\R^n)$ be the unique solution of 
	\begin{align*}
	\begin{cases}
	(-\Delta)^s u + w(x)u=F & \text{ in }\Omega, \\
	u=f &\text{ in }\Omega_e.
	\end{cases}
	\end{align*} 
	Suppose that $0\leq F\in L^\infty(\Omega)$ in $\Omega$ and $0\leq f \in L^\infty(\Omega_e)$ with $f\not \equiv 0$ in $\Omega_e$. Then $u>0$ in $\Omega$. 
\end{prop}

\section{Proof of Theorem \ref{Main Thm 1}}\label{Sec 3}
Using the results of Section~\ref{Sec 2}, we proceed to show the main theorem. 
Let $u=u(x;\varepsilon)$ be the solution to the exterior boundary value problem
\begin{align}\label{Dirichlet problem in Section 3}
\begin{cases}
(-\Delta)^s u + q(x,u,\nabla^t u) + a (x, u)=0 & \text{ in }\Omega, \\
u = \eps  f & \text{ in }\Omega_e.
\end{cases}
\end{align} 
Recall that 
$$
    q(x,u,\nabla^t u) = b(x)h(x;u)+\psi(x;d,u),
$$
where $h$ and $\psi$ are defined in \eqref{definition h} and \eqref{definition psi}, respectively.

For notational simplicity, we denote the $k^\text{th}$ derivative of $u$ with respect to $\varepsilon$ by
$$ \p^k_\eps u (x;\eps):=\frac{\p^{k} u}{\p \eps^k}(x;\eps),$$
and at $\eps=0$ we simply denote
$$
u^{(k)}(x) := \p^k_\eps|_{\varepsilon=0} u(x;\varepsilon).
$$	
 
By the UCP, we obtain the following result. 
\begin{prop}\label{prop:induction}
Let $0<t<s<1$ and let $\Omega \subset \R^n$, $n\geq 1$, be a bounded domain with smooth boundary. Let $\eps$ be a small parameter and let $f \in C_c^\infty (W_1)$.  
For $j=1,2$, consider $b_j$, $d_j$, and $a_j$ satisfying \eqref{condition bd} - \eqref{Taylor series a}, and let $u_j$ denote the solution to \eqref{Dirichlet problem in Section 3} with $b$, $d$, and $a$ replaced by $b_j$, $d_j$, and $a_j$, respectively. 

Suppose that 
\begin{align}\label{prop:a12}
\left.\Lambda_1(f)\right|_{W_2} = \left.\Lambda_2(f)\right|_{W_2}\qquad \text{ for any } f\in C^\infty_c(W_1).
\end{align}
Then 
\begin{equation}\label{prop:subinduction}
    u_1^{(k)} = u_2^{(k)} \qquad \text{ in }\R^n \qquad \text{ for all } k\in \N.
\end{equation}
\end{prop}
\begin{proof}
For clarity, we present the proof in the case $m=2$ in the nonlinear terms $\psi$ and $a$. The proof for more general $m>2$ follows a similar outline. \\
Fixing arbitrary positive integer $N$, it is sufficient to show that $u_1^{(k)} = u_2^{(k)}$ in $\R^n$ for all $1\leq k\leq N$. 

We first apply the operator $\p_\varepsilon|_{\varepsilon=0}$ to \eqref{Dirichlet problem in Section 3}. Using that $u(x;0)=0$ by well-posedness of \eqref{Dirichlet problem in Section 3}, we obtain
\begin{align}\label{Prop:equation of 1st linearization in thm 1}
\begin{cases}
(-\Delta)^s u_j^{(1)} =0 & \text{ in }\Omega,\\
u_j ^{(1)} =f &\text{ in }\Omega_e.
\end{cases}
\end{align}
Since $u_1^{(1)}=u_2^{(1)}=f $ in $\Omega_e$, the well-posedness of the problem (Theorem~\ref{Thm:well posedness}) implies that
\begin{align}\label{Prop:v_1 =v_2a1 Rn 1}
u_1^{(1)}=u_2^{(1)} =: u^{(1)} \qquad\text{ in } \R^n.
\end{align}

Next we apply $\p^2_\varepsilon|_{\varepsilon=0}$ to \eqref{Dirichlet problem in Section 3} to obtain  
\begin{align}\label{equation u2}
\begin{cases}
	(-\Delta)^s u_j^{(2)} + b_j(x)h(x;u^{(1)}) =0& \text{ in }\Omega,\\
	u_j ^{(2)} =0 &\text{ in }\Omega_e.
\end{cases}
\end{align} 
Since \eqref{prop:a12} holds, we have $(-\Delta)^s u_1^{(2)}=(-\Delta)^s u_2^{(2)}$ in $W_2$. Combining this fact with $u_1^{(2)} =u_2^{(2)}=0$ in $\Omega_e$, the UCP implies that 
\begin{align}\label{Prop:v_1 =v_2a1 Rn 2}
u_1^{(2)}=u_2^{(2)} \qquad\text{ in } \R^n.
\end{align}

Recalling that we have set $m=2$ in $\psi$, we next apply $\p^3_\varepsilon|_{\varepsilon=0}$ to \eqref{Dirichlet problem in Section 3} to obtain
\begin{align}\label{equation u3}
\begin{cases}
	(-\Delta)^s u_j^{(3)} + 2b_j(x)h(x;u^{(1)}, u^{(2)}) + 2 \psi(x;d_j,u^{(1)}) +\p^3_{z}a_j(x,0)\left(u^{(1)}\right)^3=0& \text{ in }\Omega,\\
	u_j ^{(3)} =0 &\text{ in }\Omega_e.
\end{cases}
\end{align} 
Again, since $(-\Delta)^s u_1^{(3)}=(-\Delta)^s u_2^{(3)}$ in $W_2$ by \eqref{prop:a12} and $u_1^{(3)}=u_2^{(3)}=0$ in $\Omega_e$, by the UCP we have
\begin{align*}
u_1^{(3)}=u_2^{(3)} \qquad\text{ in } \R^n.
\end{align*}

Following similar steps to above, for any $N>3$, we perform $\p^N_\varepsilon|_{\varepsilon=0}$ on \eqref{Dirichlet problem in Section 3}, which gives 
	\begin{align}\label{equ 7a1 in 1st example}
	\begin{split}
 	(-\Delta)^s u_j^{(N)} + R_{N-1}(u_j,a_j,b_j,d_j) + \p_z^{N} a_j(x,0)\left( u_j^{(1)}\right)^N=0 \qquad\text{ in }\Omega, 
	\end{split}
	\end{align}
with boundary data 
\begin{align}\label{bdryK}
u_1^{(N)}=u_2^{(N)}=0 \qquad  \hbox{ in }\Omega_e.
\end{align}
Here $R_{N-1}(u_j,a_j, b_j, d_j)$ stands for a polynomial consisting of the functions $b_j(x)$, $d_j(x,y)$, and $\p_z^{\beta}a_j(x,0)$ for $3\leq\beta\leq N-1$ and $u_j^{(k)}(x)$ for all $1\leq k\leq N-1$. Similarly, since $(-\Delta)^s u_1^{(N)}=(-\Delta)^s u_2^{(N)}$ in $W_2$ and \eqref{bdryK} hold, the UCP yields that $u_1^{(N)} = u_2^{(N)} $ in $\R^n$. The proof is complete. 
\end{proof}

With Proposition~\ref{prop:induction}, we are now ready to show the main result. The outline of the proof of Theorem~\ref{Main Thm 1} is as follows. We will first show that $b_1=b_2$ and then $\p^3_{z}a_1(x,0) = \p^3_{z}a_2(x,0)$. Using these equalities, we can show $d_1\cdot(x-y)=d_2\cdot(x-y)$. Finally, to fully recover $a$, we rely on an induction argument.  

\vskip.2cm

\begin{proof}[Proof of Theorem~\ref{Main Thm 1}]
We again present the proof for the case $m=2$ in the nonlinear terms $\psi$ and $a$. For more general $m>2$, the proof can be shown in a similar manner. 
	
 The proof is completed in 3 steps.

\noindent\textbf{Step 1. Recover $b$.} 
Let $\eps$ be sufficiently small and let $f\in C_c^\infty (W_1)$ be a non-constant function.  
For $j=1,2$, let $u_j$ be the solution to the following exterior boundary value problem:
\begin{align}\label{Dirichlet problem in Section 3_1}
\begin{cases}
(-\Delta)^s u_j + b_j(x)h(x;u_j) + \psi(x;d_j,u_j)+ a_j (x, u_j)=0 & \text{ in }\Omega, \\
u_j =\eps  f  & \text{ in }\Omega_e.
\end{cases}
\end{align}
Since \eqref{DN map in Thm 1} holds, by Proposition~\ref{prop:induction}, we have
 \begin{align}\label{thm:subinduction}
u^{(k)}:= u_1^{(k)} = u_2^{(k)} \text{ in }\R^n,\qquad k\geq 1.
\end{align}

Recall that $u_j^{(2)}$ satisfies \eqref{equation u2} for $j=1,2$. Since $u_1^{(2)}=u_2^{(2)}$, we have
\begin{align}\label{b identity}
    (b_1-b_2)(x) h(x;u^{(1)})=0\qquad\hbox{ in }\Omega,
\end{align}
where $u^{(1)}$ is the solution to \eqref{Prop:equation of 1st linearization in thm 1} with $u^{(1)}|_{\Omega_e}=f$, non-constant.
Note that by the definition of $h$, 
$$
h(x;u^{(1)}) ={c_{n,t}\over 2}\int_{\R^n}{|u^{(1)}(x)-u^{(1)}(y)|^2\over|x-y|^{n+2t}}\,dy\geq 0\qquad \hbox{ for all }x\in\Omega.
$$
We will show that in fact $h>0$ in $\Omega$. By contradiction, suppose that $h(x_0;u^{(1)})=0$ for some point $x_0\in\Omega$. This implies that $u^{(1)}\equiv u^{(1)}(x_0)$ in $\R^n$, which contradicts that the chosen exterior data $f$ is not a constant function. Therefore $h(x;u^{(1)})\neq 0$ for every point $x$ in $\Omega$.
Thus \eqref{b identity} implies that 
$$
    b_1=b_2\qquad \hbox{ in }\Omega.
$$

\noindent\textbf{Step 2. Recover $d$ and $\p^3_za(x,0)$.}
We will use that $b:=b_1=b_2$. 

In this step, we also let $\eps$ be sufficiently small and $f$ be any function in $C_c^\infty (W_1)$. 
For $j=1,2$, we also let $u_j$ be the solution to the following exterior boundary value problem:
\begin{align}\label{Dirichlet problem in Section 3_2}
\begin{cases}
(-\Delta)^s u_j + b(x)h(x;u_j) + \psi(x;d_j,u_j)+ a_j (x, u_j)=0 & \text{ in }\Omega, \\
u_j =\eps  f  & \text{ in }\Omega_e.
\end{cases}
\end{align}

Recalling \eqref{prop:subinduction}, based on \eqref{equation u3} again, the third-order linearization of \eqref{Dirichlet problem in Section 3_2} then gives 
\begin{align}\label{3 linearization}
\begin{cases}
    (-\Delta)^s u^{(3)}  + 2b(x) h(x;u^{(1)}, u^{(2)})+ 2 \psi(x;d_j,u^{(1)}) +\p^3_{z}a_j(x,0)\left(u^{(1)}\right)^3 = 0 & \text{ in }\Omega,\\
    u^{(3)} = 0 &\text{ in }\Omega_e.
\end{cases}
\end{align}
Subtracting \eqref{3 linearization} with $j=2$ from \eqref{3 linearization} with $j=1$, we obtain
\begin{align}\label{id da}
(u^{(1)})^2(x)\left(2\int_{\R^n} (d_1-d_2)(x,y) \cdot\nabla^t u^{(1)}(x,y) \, dy + (\p^3_{z}a_1(x,0) - \p^3_{z}a_2(x,0)) u^{(1)}(x) \right) = 0.
\end{align} 
Here $u^{(1)}$ is the solution to \eqref{Prop:equation of 1st linearization in thm 1} with $u^{(1)}|_{\Omega_e}=f$ for any $f\in C^\infty_c(W_1)$.
By the Runge approximation property (see \cite[Lemma~5.1]{ghosh2016calder} with $q=0$), we may find a sequence of solutions $v_k$ to \eqref{Prop:equation of 1st linearization in thm 1} 
such that $v_k\to 1$ in $L^2(\Omega)$ as $k\rightarrow \infty$. Then there is a subsequence $v_{k_j}$, which converges pointwise almost everywhere (a.e.) to $1$ as $j\rightarrow \infty$. 
Note, then, that since we assume $(d_1-d_2)(x,y)|x-y |^{-n/2-t}\in L^2(\Omega)$ for any fixed $x\in\Omega$, we have that $\int_{\R^n} (d_1-d_2)(x,y) \cdot\nabla^t v_{k_j}(x,y) \, dy \to 0$ as $j\to \infty$. 
Replacing $u^{(1)}$ by $v_{k_j}$ in \eqref{id da} and taking $j\to \infty$, the first term thus vanishes, yielding 
\begin{align*} 
     \p^3_{z}a_1(x,0) = \p^3_{z}a_2(x,0).
\end{align*}

With this, we now turn back to \eqref{id da} and get that
 \begin{align}\label{recover da 1}
  (u^{(1)})^2(x) \int_{\R^n} (d_1-d_2)(x,y) \cdot (y-x) {u^{(1)}(x)-u^{(1)}(y)\over |x-y|^{n/2+t+1}} \, dy  =0.
 \end{align} 
For any fixed $x_0\in\Omega$, since $(d_1-d_2)(x_0,y) \cdot (y-x_0)$ is continuous in $\Omega$, we may define the following two open subsets of $\Omega$:
$$
    A_+ := \{ y\in\Omega\setminus\{x_0\}:\,  (d_1-d_2)(x_0,y) \cdot (y-x_0)> 0 \}
$$
and 
$$
    A_- := \{ y\in\Omega\setminus\{x_0\}:\,  (d_1-d_2)(x_0,y) \cdot (y-x_0) <0 \}.
$$
We will show by contradiction that $(d_1-d_2)(x,y) \cdot (y-x) = 0$. Suppose that at least one of $A_\pm$ is not empty.

We define the function $\varphi_{x_0}$ by
$$
\varphi_{x_0}(y) = \begin{cases}
{1\over 1+|x_0-y|^2} & \hbox{ if } y\in A_+, \\
{1+2|x_0-y|^2\over 1+|x_0-y|^2} &\hbox{ if } y\in A_-,\\
1 & \hbox{ if }y\in\Omega\setminus(A_+\cup A_-).
\end{cases}
$$
Since $\Omega$ is bounded, $\varphi_{x_0}$ is in $L^2(\Omega)$. It is clear that $\varphi_{x_0}(x_0)=1$ since $x_0\notin A_\pm$. Then we have
\begin{align*} 
\left\{
\begin{array}{cc}
    \varphi_{x_0}(x_0)=1 > \varphi_{x_0}(y) & \hbox{ for all }y\in A_+,\\
    \varphi_{x_0}(x_0)=1 < \varphi_{x_0}(y) & \hbox{ for all }y\in A_-,\\
\end{array} 
\right.
\end{align*} 
and thus
\begin{align}\label{function varphi}
    (d_1-d_2)(x_0,y) \cdot (y-x_0)  {\varphi_{x_0}(x_0) - \varphi_{x_0}(y) \over |x_0-y|^{n/2+t+1}} 
    > 0\qquad \hbox{ for all }y\in A_\pm.
\end{align} 

Again by the Runge approximation property, there exists a sequence of solutions $\tilde v_k$ to \eqref{Prop:equation of 1st linearization in thm 1} 
such that $\tilde v_k \to \varphi_{x_0}$ in $L^2(\Omega)$ as $k\rightarrow \infty$, which implies that there exists a subsequence $\tilde v_{k_j}\to \varphi_{x_0}$ a.e. as $j\rightarrow \infty$. Since  $(d_1-d_2)(x,y)|x-y |^{-n/2-t}\in L^2(\Omega)$ for any fixed $x\in\Omega$, we may replace $u^{(1)}$ by $\tilde v_{k_j}$ in \eqref{recover da 1} and take $j\rightarrow \infty$ to obtain
 \begin{align}\label{recover da 2}
\varphi_{x_0}^2(x_0) \int_{\R^n} (d_1-d_2)(x_0,y) \cdot (y-x_0)  {\varphi_{x_0}(x_0) - \varphi_{x_0}(y) \over |x_0-y|^{n/2+t+1}}\, dy  =0. 
\end{align} 
However, since $0\neq \varphi_{x_0}(x_0)$, by \eqref{condition bd} and \eqref{function varphi}, the integral in \eqref{recover da 2} must be strictly positive for any nonempty $A_\pm$, which is a contradiction. Therefore, both $A_\pm$ must be empty sets, which implies that 
$$
d_1(x_0,y)\cdot(x_0-y)=d_2(x_0,y)\cdot(x_0-y) \qquad \hbox{ for all }y\in \Omega.
$$
Since $x_0\in\Omega$ is arbitrary, we thus have
$$
d_1(x,y)\cdot(x-y)=d_2(x,y)\cdot(x-y) \qquad\hbox{for each $(x,y)\in\Omega\times \Omega$}.
$$  
Thus we uniquely determine the $(x-y)$-direction component of $d(x,y)$.  
 
Now the problem boils down to showing the uniqueness of the potential $a$. It is then sufficient to show that $\p_z^k a_1(x,0)=\p_z^ka_2(x,0)$ for $k>3$.\\

\noindent\textbf{Step 3. Recover higher order terms $\p^k_z a(x,0)$, $k>3$.}
Step 1 and Step 2 have shown that 
\begin{align}\label{thm 3 unique}
 b_1=b_2,\ \psi(x;d_1,u^{(1)}) = \psi(x;d_2,u^{(1)}),\ \p_z^3 a_1(x,0)=\p_z^3a_2(x,0). 
\end{align}
By induction, for any fixed $N\in\mathbb{N}$, suppose that 
\begin{align}\label{thm 3 unique 1}
\p_z^j a_1(x,0)=\p_z^ja_2(x,0) \qquad \hbox{ for }3 \leq j \leq N-1.
\end{align}
It is sufficient to show that
$\p_z^{N}a_1(x,0)=\p_z^{N}a_2(x,0)$ holds as well. From now on, we will use $j$ subscripts on $a_j$ only since the coefficients $b$, $d$ have been recovered. 

Recall from \eqref{equ 7a1 in 1st example} that  
\begin{align}\label{equ 7a1 in 1st example 1}
\begin{split}
(-\Delta)^s u^{(N)} + R_{N-1}(u ,a_j,b ,d ) + \p_z^{N} a_j(x,0)\left( u ^{(1)}\right)^N=0 \qquad\text{ in }\Omega 
\end{split}
\end{align}
with boundary data 
\begin{align}\label{bdryK 1}
u^{(N)}=u_1^{(N)}=u_2^{(N)}=0 \qquad  \hbox{ in }\Omega_e,
\end{align}
where $R_{N-1}(u ,a_j, b , d )$ stands for a polynomial consisting of the functions $b (x)$, $d (x,y)$, and $\p_z^{\beta}a_j(x,0)$ for $3\leq\beta\leq N-1$ and $u ^{(k)}(x)$ for all $1\leq k\leq N-1$.
Note that \eqref{thm 3 unique 1} implies that 
$$R_{N-1}(u ,a_1, b , d )= R_{N-1}(u ,a_2, b , d ),$$
and therefore \eqref{equ 7a1 in 1st example 1} gives
$$
\p_z^{N} a_1(x,0)\left( u ^{(1)}\right)^N=\p_z^{N} a_2(x,0)\left( u^{(1)}\right)^N.
$$
Choosing exterior data $f>0$ in \eqref{Prop:equation of 1st linearization in thm 1} and using the maximum principle (Proposition~\ref{Prop: strong max principle}), we have $u^{(1)}\neq 0$. This gives $\p_z^{N} a_1(x,0)=\p_z^{N} a_2(x,0)$. Finally, by the uniqueness of the expansion \eqref{Taylor series a}, we obtain $a_1(x,z)=a_2(x,z)$. The proof is complete.	
\end{proof}

\bigskip
\bigskip
\noindent\textbf{Acknowledgment.}
R.-Y. Lai was partially supported by NSF grant DMS-1714490 and DMS-2006731.
L. Ohm was supported by NSF grant DMS-1714490 during summer 2020 and NSF grant DMS-2001959.

\bigskip
\bibliographystyle{abbrv}
\bibliography{NonlinearNonlocal}
\end{document}